\documentclass[11pt]{amsart}

\usepackage{amsmath,amssymb,amsthm,comment,mathtools,tikz,tikz-cd,enumerate,mathtools}
\usetikzlibrary{calc}
\usepackage{hyperref}
\usepackage[margin=1.25in]{geometry}
\usepackage[all]{hypcap}
\usetikzlibrary{graphs,graphs.standard}
\usepackage{longtable}
\newtheorem{theorem}{Theorem}[section]

\newtheorem{corollary}[theorem]{Corollary}
\newtheorem{lemma}[theorem]{Lemma}
\newtheorem{question}[theorem]{Question}
\newtheorem{proposition}[theorem]{Proposition}
\newtheorem*{genericthm*}{\thistheoremname}
\newenvironment{namedthm*}[1]
  {\renewcommand{\thistheoremname}{#1}
   \begin{genericthm*}}
  {\end{genericthm*}}
\theoremstyle{definition}
\newcommand{\thistheoremname}{}

\usepackage{comment}

\theoremstyle{remark}

\newtheorem*{remark}{Remark}

\DeclareMathOperator{\diag}{diag}

\newcommand{\vd}{\mathbf{d}}

\newcommand{\vr}{\mathbf{r}}

\title[Spectral Radii of Arithmetical Structures on Cycle Graphs]{Spectral Radii of Arithmetical Structures on Cycle Graphs}

\author{Alexander Diaz-Lopez}
\address[A.~Diaz-Lopez]{Department of Mathematics and Statistics, Villanova University, 800 Lancaster Ave (SAC 305), Villanova, PA 19085, USA}
\email{alexander.diaz-lopez@villanova.edu}
\author{Kathryn Haymaker}
\address[K.~Haymaker]{Department of Mathematics and Statistics, Villanova University, 800 Lancaster Ave (SAC 305), Villanova, PA 19085, USA}
\email{kathryn.haymaker@villanova.edu}
\author{Michael Tait}
\address[M.~Tait]{Department of Mathematics and Statistics, Villanova University, 800 Lancaster Ave (SAC 305), Villanova, PA 19085, USA}
\email{michael.tait@villanova.edu}

\begin{document}
\begin{abstract}
    Let $G$ be a finite, connected graph. An arithmetical structure on $G$ is a pair of positive integer-valued vectors $(\vd,\vr)$ such that $(\diag(\vd)-A_G)\cdot \vr=\textbf{0},$ where the entries of $\vr$ have $\gcd$ 1 and $A_G$ is the adjacency matrix of $G$. In this article we find the arithmetical structures that maximize and minimize the spectral radius of $(\diag(\vd)-A_G)$ among all arithmetical structures on the cycle graph $\mathcal{C}_n.$
\end{abstract}
\maketitle

\section{Introduction}

Given a finite connected graph $G$ with vertex set $V$ of cardinality $n$, an arithmetical structure on $G$ is a tuple of vectors $(\vd,\vr) \in \mathbb{Z}^n_{>0}\times \mathbb{Z}^n_{>0}$ such that 
\begin{equation}\label{eq:arithdefinition}
(\diag(\vd)-A_G)\cdot \vr = \textbf{0},
\end{equation}
where the entries of $\vr$ have $\gcd$ 1 and $A_G$ is the adjacency matrix of $G$; when no confusion arises, we will denote $A_G$ simply as $A$. Using the same ordering of the vertices of $G$ used in the adjacency matrix $A$, we can think of the entries of $\vr=(r_v)_{v\in V}$ as labels on the vertices.  Equation \eqref{eq:arithdefinition} states that, for every vertex $v$, the entry $r_v$ must divide the sum of the labels of adjacent vertices. More specifically, $r_vd_v=\sum_{w\sim v} r_{w}$, where $w\sim v$ denotes that $w$ and $v$ are adjacent vertices and $d_v$ are the entries of the $\vd$ vector. 

Arithmetical structures were introduced by Lorenzini \cite{L89} to study degeneration of curves in algebraic geometry. Lorenzini \cite[Lemma 1.6]{L89} used a non-constructive argument to show that the number of arithmetical structures on a finite connected graph $G$ is finite. Since then, several groups \cite{A20,A23,B18,C18,C17,GW,V21} have studied arithmetical structures on different families of graphs such as path graphs, cycle graphs, bidents, star graphs, complete graphs, $E_n$-graphs, and paths with a double edge. 
In general, it has proven difficult to find closed formulas for the number of arithmetical structures on a graph. Full enumeration results (including a closed formula) are only known for path graphs and cycle graphs \cite{B18}. For star graphs \cite[Corollary 3.1]{C17} and complete graphs \cite[Theorem 5.1]{C18}, the number of arithmetical structures is enumerated by variants of Egyptian fractions for which we have no known closed formulas.

For an arithmetical structure $(\vd,\vr)$ on a graph $G$, the matrix $(\diag(\vd)-A)$ is positive semidefinite and has rank $n-1$ \cite[Proposition 1.1]{L89}. Since the nullity of this matrix is $1$ and the entries of $\vr$ are positive and have gcd $1$, the vector $\vd$ uniquely determines the vector $\vr$. We will denote the matrix $(\diag(\vd)-A)$ by $L(G, \vd)$. These matrices generalize the  Laplacian matrix of $G$, which is the matrix $\diag(\mathrm{deg}(G))-A$ where $\mathrm{deg}(G)$ is the vector recording the degrees of the vertices of $G$. The pair $(\mathrm{deg}(G),\textbf{1})$ is an arithmetical structure and is called the \textit{Laplacian arithmetical structure} because it corresponds to the Laplacian matrix of a graph. Laplacian matrices of graphs are well-studied and several surveys have been written about them (for example \cite{MerrisSurvey1, MerrisSurvey2}) and their eigenvalues (for example \cite{Das, GMS, mohar, zhang}).  Famous applications of graph Laplacian eigenvalues include Kirchoff's theorem \cite{kirchhoff} from which one obtains the number of spanning trees in a graph, and theorems measuring connectivity via the second smallest eigenvalue, for example \cite{fiedler1, fiedler2}.

Given a real and symmetric matrix $M$, the spectral radius of $M$ is its largest eigenvalue. In the case of the Laplacian arithmetical structure, many papers have been written about the spectral radius of the Laplacian matrix $(\diag(\vd)-A)$ (for example \cite{guo, guo2, HZ, LSC, LLT, shi, YLT}). In \cite{W21}, Wang and Hou studied the spectral radius of the generalized Laplacian matrix $\diag(\vd)-A$ associated to arithmetical structures on the path graph $\mathcal{P}_{n}$. In particular, they show that among all structures on $\mathcal{P}_{n}$, the Laplacian arithmetical structure has the minimal spectral radius \cite[Theorem 4.1]{W21} and described the structure that maximizes the spectral radius \cite[Theorem 4.12]{W21}. In this paper, we consider the analogous questions for arithmetical structures on the cycle graph $\mathcal{C}_n$. We show that among all arithmetical structures on $\mathcal{C}_n$, the Laplacian arithmetical structure $(\mathrm{deg}(G),\textbf{1})$ minimizes the spectral radius (Theorem \ref{thm:minimal}) and the structure given by $\vd=(1,n+2,2,2,\ldots, 2)$ and $\vr=(n,1,2,3, \ldots, n-1)$ maximizes the spectral radius (Theorem \ref{thm:main}).

\section{Background}

\subsection{Arithmetical Structures on cycle graphs}\label{subsec:arith}
When $G$ is the cycle graph $\mathcal{C}_n$ with vertex set $\{v_1,v_2,\ldots,v_n\}$, Equation $\eqref{eq:arithdefinition}$ becomes 
$$\begin{pmatrix}
d_1&-1&0&0&\ldots&-1\\
-1&d_2&-1&0&\ldots&0\\
0&-1&d_3&-1&\ldots&0\\
0&0&\ddots&\ddots&\ddots&0\\
0&\ldots&0&-1&d_{n-1}&-1\\
-1&0&\ldots&0&-1&d_n\\
\end{pmatrix}\cdot \begin{pmatrix}
r_1\\r_2\\r_3\\\vdots\\ r_{n-1}\\r_n
\end{pmatrix} =\textbf{0},$$
or equivalently, $r_id_i=r_{i-1}+r_{i+1}$ for all $i\in[n]$, where the indices are taken modulo $n$. We associate the entries $r_i$ and $d_i$ with the vertex $v_i$ and consider them as labels on the vertex. Since the matrix $\diag(\vd)-A$ has rank $n-1$ \cite[Proposition 1.1]{L89}, the values of $\vd$ completely determine the structure $(\vd,\vr)$ as $\vr$ is the unique vector in the kernel of $\diag(\vd)-A$ with positive entries having $\gcd$ 1. Similarly, the vector $\vr$ determines the vector $\vd$ via the equations $d_i=(r_{i-1}+r_{i+1})/r_i$. Thus, we could define arithmetical structures on $\mathcal{C}_n$ as $\vr$-labels on the vertices of $G$ such that $r_i$ divides the sum of the $\vr$-labels of its two neighbors, 
and refer to them via their $\vr$-labels or via their $\vd$-labels. 

The vertices of  $\mathcal{C}_n$ can be thought of as the vertices of a regular $n$-gon. Under this view, we can apply any symmetry of the dihedral group $D_{2n}$ to these vertices and each vertex preserves its set of neighbor vertices. Thus, if $(\vd,\vr)$ is an arithmetical structure on $\mathcal{C}_n$, we can permute the entries of both $\vd$ and $\vr$ by any element in $D_{2n}$ and still get an arithmetical structure on $\mathcal{C}_n$. In Table \ref{tab:struc}, we show  the $\vd$ vector of  the arithmetical structures on $\mathcal{C}_n$ for $n=3,4,5,6$, up to symmetries given by the action of $D_{2n}$.
\begin{figure}
    \centering
\begin{tikzpicture}[
    tlabel/.style={pos=0.4,right=-1pt},
    baseline=(current bounding box.center),scale=1.5 ]
\node (4) at (0,0) [circle,draw=black, label=above right:{$(2,3)$},inner sep=2pt, minimum size=.2cm]{$v_4$};
\node (3) at (120:1) [circle,draw=black, label=above:{$(2,2)$},inner sep=2pt, minimum size=.2cm]{$v_3$};
\node (2) at ($(120:1)+(180:1)$) [circle,draw=black, label=above left:{$(8,1)$},inner sep=2pt, minimum size=.2cm]{$v_2$};
\node (1) at ($(120:1)+(180:1)+(240:1)$) [circle,draw=black, label=left:{$(1,6)$},inner sep=2pt, minimum size=.2cm]{$v_1$};
\node (6) at ($(120:1)+(180:1)+(240:1)+(300:1)$) [circle,draw=black, label=below left:{$(2,5)$},inner sep=2pt, minimum size=.2cm]{$v_6$};
\node (5) at ($(120:1)+(180:1)+(240:1)+(300:1)+(0:1)$) [circle,draw=black, label=below:{$(2,4)$},inner sep=2pt, minimum size=.2cm]{$v_5$};
\draw (1)--(2)--(3)--(4)--(5)--(6)--(1);
\end{tikzpicture}\begin{tikzcd}[scale=2,every label/.append
  style={font=\large}]
 \arrow[r,shift left,"\text{smoothing}"]
    &[3em]
 \arrow[l,shift left,"\text{subdivision}"]
\end{tikzcd}
\begin{tikzpicture}[
    tlabel/.style={pos=0.4,right=-1pt},
    baseline=(current bounding box.center),scale=1.5 ]
\node (4) at (0,0) [circle,draw=black, label=above right:{$(2,3)$},inner sep=2pt, minimum size=.2cm]{$v_4$};
\node (3) at (120:1) [circle,draw=black, label=above:{$(2,2)$},inner sep=2pt, minimum size=.2cm]{$v_3$};
\node (2) at ($(120:1)+(180:1)$) [circle,draw=black, label=above left:{$(7,1)$},inner sep=2pt, minimum size=.2cm]{$v_2$};
\node (6) at ($(120:1)+(180:1)+(240:1)+(300:1)$) [circle,draw=black, label=below left:{$(1,5)$},inner sep=2pt, minimum size=.2cm]{$v_6$};
\node (5) at ($(120:1)+(180:1)+(240:1)+(300:1)+(0:1)$) [circle,draw=black, label=below:{$(2,4)$},inner sep=2pt, minimum size=.2cm]{$v_5$};
\draw (2)--(3)--(4)--(5)--(6)--(2);
\end{tikzpicture}
    \caption{On the left, we have an arithmetical structure on $\mathcal{C}_6$ with vertices labeled as $(d_i,v_i)$. On the right, we have the structure obtained by smoothing the left structure at $v_1$.}
    \label{fig:subdivisionsmoothing}
\end{figure}
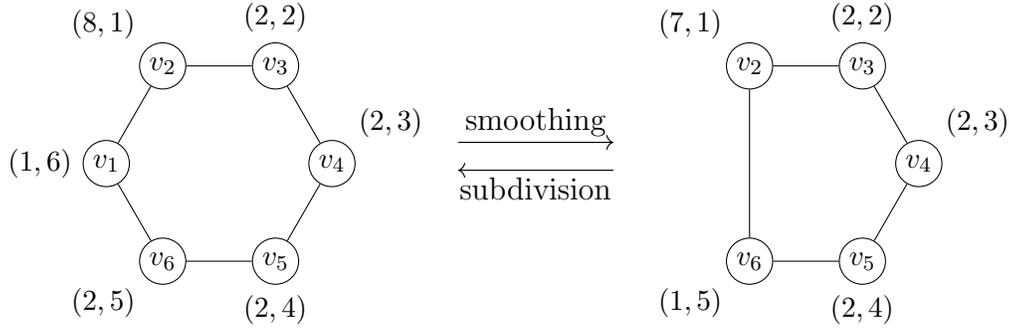

\begin{table}[]\label{tab:struc}
    \centering
    \begin{tabular}{|c|c|}\hline
         $n$&$\vd$ vector of arithmetical structures of $\mathcal{C}_n$ (up to symmetry)  \\\hline
         3& $(2,2,2)_{3.00}$, $(3,3,1)_{4.00}$, $\mathbf{(1,5,2)_{5.41}}$\\ \hline
         4& $(2,2,2,2)_{4.00}$, $(3,2,3,1)_{4.41}$, $(4,3,2,1)_{5.00}$, $(4,1,4,1)_{5.00}$, \\&
         $(5,3,1,2)_{5.73}$, $(6,1,3,1)_{6.41}$, $\mathbf{(1,6,2,2)_{6.45}}$
         \\\hline
         5 & $(2, 2, 2, 2, 2)_{3.62}$,
 $(3, 2, 2, 3, 1)_{4.00}$, $(1, 4, 1, 3, 3)_{4.62}$,\\&
 $(1, 4, 2, 3, 2)_{4.73}$, 
 $(4, 3, 3, 1, 2)_{5.00}$,
 $(1, 4, 4, 1, 3)_{5.24}$,\\&
 $(5, 1, 2, 4, 1)_{5.49}$, 
 $(3, 5, 1, 2, 2)_{5.64}$, 
 $(5, 3, 2, 1, 3)_{5.76}$,\\&
 $(1, 2, 2, 5, 4)_{5.87}$, 
 $(2, 1, 4, 1, 6)_{6.43}$,
 $(6, 2, 1, 3, 2)_{6.47}$,\\&
 $(1, 3, 6, 1, 3)_{6.49}$, 
 $(3, 1, 7, 1, 2)_{7.33}$, 
 $\mathbf{(1, 7, 2, 2, 2)_{7.35}}$\\ \hline
 6&$(2, 2, 2, 2, 2, 2)_{4.00}$,
$(3, 2, 2, 2, 3, 1)_{4.30}$,
$(3, 3, 1, 3, 3, 1)_{4.56}$,\\&
$(3, 2, 3, 1, 4, 1)_{4.73}$,
$(4, 2, 2, 3, 2, 1)_{4.76}$,
$(4, 2, 3, 3, 1, 2)_{4.90}$,\\&
$(4, 2, 4, 1, 3, 1)_{5.00}$,
$(4, 3, 1, 4, 2, 1)_{5.00}$,
$(4, 1, 4, 1, 4, 1)_{5.00}$,\\&
$(4, 3, 3, 2, 1, 3)_{5.16}$,
$(4, 3, 4, 1, 2, 2)_{5.22}$,
$(4, 4, 1, 4, 1, 2)_{5.33}$,\\&
$(3, 3, 1, 5, 1, 2)_{5.50}$,
$(5, 2, 3, 2, 2, 1)_{5.56}$,
$(5, 1, 4, 3, 1, 2)_{5.58}$,\\&
$(5, 2, 1, 5, 2, 1)_{5.65}$,
$(5, 3, 3, 1, 3, 1)_{5.68}$,
$(5, 1, 5, 1, 3, 1)_{5.68}$,\\&
$(5, 3, 2, 3, 1, 2)_{5.71}$,
$(1, 2, 3, 5, 3, 3)_{5.82}$,
$(5, 4, 1, 3, 2, 1)_{5.84}$,\\&
$(5, 4, 2, 1, 3, 2)_{5.91}$,
$(5, 3, 2, 2, 1, 4)_{5.95}$,
$(5, 4, 1, 3, 1, 3)_{5.95}$,\\&
$(5, 5, 1, 2, 2, 2)_{6.23}$,
$(4, 2, 1, 6, 1, 2)_{6.39}$,
$(6, 1, 4, 2, 2, 1)_{6.40}$,\\&
$(6, 1, 5, 1, 2, 2)_{6.45}$,
$(6, 2, 4, 1, 2, 2)_{6.48}$,
$(6, 2, 1, 5, 1, 2)_{6.48}$,\\&
$(6, 3, 2, 2, 2, 1)_{6.50}$,
$(6, 3, 2, 1, 4, 1)_{6.50}$,
$(6, 1, 4, 2, 1, 3)_{6.50}$,\\&
$(1, 2, 6, 3, 1, 4)_{6.53}$,
$(1, 2, 3, 2, 6, 3)_{6.54}$,
$(6, 4, 1, 2, 3, 1)_{6.61}$,\\&
$(7, 1, 4, 1, 3, 1)_{7.33}$,
$(7, 2, 3, 1, 3, 1)_{7.36}$,
$(2, 1, 7, 2, 1, 4)_{7.36}$,\\&
$(7, 1, 3, 3, 1, 2)_{7.36}$,
$(1, 2, 7, 2, 2, 3)_{7.39}$,
$(7, 3, 1, 3, 2, 1)_{7.40}$,\\&
$(8, 1, 3, 2, 2, 1)_{8.28}$,
$(8, 1, 2, 3, 2, 1)_{8.29}$,
$\mathbf{(1,8, 2, 2, 2, 2)_{8.30}}$
\\ \hline
    \end{tabular}
    \caption{The $\vd$ vectors of the arithmetical structures on $\mathcal{C}_n$ (up to symmetry), with subscript given by  $\mu_1(L(\mathcal{C}_n, \vd))$. The structure that maximizes $\mu_1(L(\mathcal{C}_n, \vd))$ is in bold.}
\end{table}

Braun et al. \cite[Theorem 30]{B18} showed that the total number of arithmetical structures on $\mathcal{C}_n$ is $\binom{2n-1}{n-1}$ and proved that all non-Laplacian arithmetical structures on $\mathcal{C}_n$ can be obtained from arithmetical structures on $\mathcal{C}_{n-1}$ via a process called \textit{subdivision}. This process, and its inverse process called \textit{smoothing}, play an important role for us in this paper. While these processes make sense for other graphs, we only present them here for cycle graphs. An example is shown in Figure \ref{fig:subdivisionsmoothing}.

Given an arithmetical structure $(\vd',\vr')$ on the cycle graph $\mathcal{C}_n$ and an index $i \in [n]$, we can create an arithmetical structure $(\vd,\vr)$ on a cycle graph on $n+1$ vertices by introducing a vertex $w$ between $v_i$ and $v_{i+1}$ (indices taken modulo $n$) with $\vr$-label $r_{w}=r_{i}'+r_{i+1}'$ and $\vd$-label $d_w=1$. Define $\vd$ and $\vr$ as follows
\[ 
d_v=\begin{cases}
d_v' & \text{ if } v\neq v_{i}, v\neq v_{i+1}, \text { and }  v\neq w\\
d_v'+1 &\text{ if } v= v_{i} \text { or } v= v_{i+1}\\
1 &\text{ if } v= w\\
\end{cases}
\qquad
r_v=\begin{cases}
r_v' & \text{ if } v\neq w\\
r_{i}'+r_{i+1}' &\text{ if } v=w.
\end{cases}
\]
The pair $(\vd,\vr)$ is called the subdivision of $(\vd',\vr')$ between vertices $v_i$ and $v_{i+1}$. 

The inverse operation of subdivision is called smoothing and can be performed at a vertex of the graph with associated $\vd$-label 1. Given an arithmetical structure $(\vd,\vr)$ on $\mathcal{C}_n$ and an index $i$ such that the $\vd$-label of $v_i$ is 1, consider the graph obtained by removing vertex $v_i$ and connecting $v_{i-1}$ and $v_{i+1}$ via an edge. Let $(\vd',\vr')$ be defined as follows on $\mathcal{C}_{n-1}$
\[ 
d_v'=\begin{cases}
d_v & \text{ if } v\neq v_{i-1} \text { and } v\neq v_{i+1}\\
d_v-1 &\text{ if } v= v_{i-1} \text { or } v= v_{i+1}\\
\end{cases}
\qquad
r_v'=
r_v \text{ for all vertices } v.
\]
The pair $(\vd',\vr')$ is called the smoothing of $(\vd,\vr)$ at vertex $v_i$. Corrales and Valencia \cite[Theorem 5.1]{C18} showed that these operations produce arithmetical structures on their respective resulting graphs and the operations are inverses of each other. 
\begin{remark}
Corrales and Valencia \cite{C18} defined a clique-star operation that generalizes the subdivision procedure. The clique-star removes all edges from a clique of a graph and introduce a new vertex connected to all the vertices of the clique. This clique-star operation is a special case of Lorenzini's blowup operation \cite[Page 485]{L89}, but this blowup operation is not always defined on graphs. In \cite{KR}, Keyes and Reiter present an operation that generalizes the smoothing procedure and can be applied to any vertex (without the condition that the $\vd$-label is 1). These generalized operations might be useful for future work extending our results and the results of \cite{W21} to other graphs.
\end{remark}

The following result is key in enumerating arithmetical structures on cycle graphs as it implies that non-Laplacian arithmetical structures on $\mathcal{C}_n$ can be created from structures on $\mathcal{C}_{n-1}$ via subdivision. We will make use of this result in some of our proofs. 

\begin{proposition}[{\cite[Theorem 6.5]{C18}}]\label{prop:dv=1}
If $(\vd, \vr)$ is an arithmetical structure of the cycle $\mathcal{C}_n$ with $n \geq 4$ vertices
and $\vd \neq (2,2,\cdots,2)$, then there exists a vertex $v_i$ of $\mathcal{C}_n$ such that
$d_i=1$ and $d_{i-1}\neq 1\neq d_{i+1}$, where the indices are taken modulo $n$.
\end{proposition}

\subsection{Spectral radius of matrices}
In this subsection we collect well-known results about the spectral radius of matrices that we will use. All of our matrices are real and symmetric, and so they admit an orthonormal basis of eigenvectors corresponding to real eigenvalues. A consequence of this is the Rayleigh principle, which is a consequence of the min-max theorem. This theorem is classical, see Theorem 8.8 in \cite{zhangbook} for a reference.

\begin{lemma}\label{lem:rayleigh}[Rayleigh Principle]
If $M$ is a real and symmetric matrix, then its spectral radius is given by 
\[
\max_{\mathbf{x}\not=0} \frac{\mathbf{x}^T M \mathbf{x}}{\mathbf{x}^T\mathbf{x}}. 
\]
\end{lemma}
Two corollaries of Lemma \ref{lem:rayleigh} are the following.

\begin{corollary}\label{cor:absolute value}
Let $M$ be a real symmetric matrix, and $|M|$ be the matrix obtained by taking the absolute value of each entry of $M$. Then the spectral radius of $|M|$ is at least as big as the spectral radius of $M$.
\end{corollary}
\begin{corollary}\label{cor:interlacing}
Let $M$ be a real symmetric matrix and $M_i$ the matrix with the $i$'th row and column removed. Then the spectral radius of $M$ is at least as big as the spectral radius of $M_i$.
\end{corollary}
We will combine Corollary \ref{cor:absolute value} with a well-known result bounding the spectral radius in terms of maximum row sum, see for example Theorem 5.24 in \cite{zhangbook}.

\begin{lemma}\label{prop:rowsum}
Let $M$ be a non-negative, real, symmetric matrix. Then its spectral radius is bounded above by its maximum row sum. 
\end{lemma}
Finally, we will use the Courant-Weyl inequalities (see Theorem 8.12 in \cite{zhangbook} for a reference).

\begin{lemma}[Courant-Weyl inequalities]\label{CW inequality}
Let $X$ and $Y$ be real symmetric matrices with eigenvalues $\mu_1(X) \geq \mu_2(X) \geq \cdots \geq \mu_n(X)$ and $\lambda_1(Y) \geq \lambda_2(Y) \geq \cdots \geq \lambda_n(Y)$, and let their sum have eigenvalues $\rho_1(X+Y) \geq \cdots \geq \rho_n(X+Y)$. Then for all $i$, we have
\[
\mu_i(X) + \lambda_n(Y) \leq \rho_i(X+Y) \leq \mu_i(X) + \lambda_1(Y).
\]
\end{lemma}

\section{Minimum spectral {radius}}
In this section we show that the arithmetical structure on $\mathcal{C}_n$ that minimizes the spectral radius of the matrix $L=(\diag(\vd)-A)$ is $(\vd,\vr)$ with $\vd=(2,2,\ldots, 2)$ and $\vr=(1,1,\ldots, 1)$, that is, the Laplacian arithmetical structure. {For comparison, in Table \ref{tab:struc} we show the different arithmetical structures (up to symmetries) on $\mathcal{C}_n$  for $n = 3,4,5,6$  with their spectral radius}. 
Throughout this section, given a real and symmetric matrix $M$, we will use $\mu_1(M)$ to denote its spectral radius. We start with a proposition on the spectral radius of the Laplacian arithmetical structure for the cycle which is not new but for which we provide a proof for completeness.

\begin{proposition}\label{prop:mulaplacian}
Let $\textbf{d}$ be the Laplacian arithmetical structure on $\mathcal{C}_n$ {for $n\geq 3$}. Then, $\mu_1(L(\mathcal{C}_n,\textbf{d}))=4$ if $n$ is even and $\mu_1(L(\mathcal{C}_n,\textbf{d}))<4$ if $n$ is odd.
\end{proposition}
\begin{proof}
Since {each vertex has degree $2$ in $\mathcal{C}_n$}, we have that 
\[
L(\mathcal{C}_n,\textbf{d})=\textup{diag}(\mathbf{d}) - A_{\mathcal{C}_n} = 2I - A_{\mathcal{C}_n}.
\]
It is well-known (for example Section 1.4.3 of \cite{BH}) that $A_{\mathcal{C}_n}$ has eigenvalues $2\cos(2\pi j/n)$, where $j$ ranges from $0$ to $n-1$. Therefore, the maximum eigenvalue of the Laplacian structure is given by $2-2\cos\left(2\pi \lfloor \frac{n}{2} \rfloor / n \right)$. In particular, this quantity is equal to $4$ if $n$ is even and is strictly less than $4$ if $n$ is odd.
\end{proof}

\begin{lemma}\label{lem:31322onpath}
Let $n\geq5$ and $M$ be the $n\times n$ matrix
$$M=
\begin{pmatrix}
3&-1&0&0&0&\ldots&0\\
-1&1&-1&0&0&\ldots&0\\
0&-1&3&-1&0&\ldots&0\\
0&0&-1&2&-1&\ddots&0\\
0&0&\ddots&\ddots&\ddots&\ddots&0\\
0&0&\ldots&0&-1&2&-1\\
0&0&\ldots&0&0&-1&2\\
\end{pmatrix},$$
then $\mu_1(M)>4$.
\end{lemma}
\begin{proof}
We proceed by induction on $n$. When $n=5$, a quick Sage computation shows that $\mu_1(M)>4.08$.  For $n\geq 6$, let $M_n$ denote denote the principal submatrix of $M$ with row $n$ and column $n$ deleted, that is, 
$$M_n=
\begin{pmatrix}
3&-1&0&0&\ldots&0\\
-1&1&-1&0&\ldots&0\\
0&-1&3&-1&\ldots&0\\
0&0&-1&2&\ddots&0\\
0&\ldots&0&\ddots&\ddots&-1\\
0&\ldots&0&0&-1&2\\
\end{pmatrix}.$$
By the induction hypothesis, $\mu_1(M_n)>4$. By Corollary \ref{cor:interlacing},
$\mu_1(M)\geq \mu_1(M_n)>4.$
\end{proof}

\begin{lemma}\label{lem:31322oncycle}
Given the arithmetical structure $\vd=(3,1,3,2,2,\ldots, 2)$  on $\mathcal{C}_n$, we have that $\mu_1(L(\mathcal{C}_n,\vd))=4$ if $n=3,5$ and $\mu_1(L(\mathcal{C}_n,\vd))>4$ for $n=4$ and $n\geq 6$. 
\end{lemma}
\begin{proof}
For $n=3,4,5$, we prove it directly by computing the spectral radius using Sage,
$$\mu_1(L(\mathcal{C}_3,(3,1,3)))=4, \quad\mu_1(L(\mathcal{C}_4,(3,1,3,2)))\approx 4.41421, \quad\mu_1(L(\mathcal{C}_5,(3,1,3,2,2)))=4. $$
{For $n\geq 6$}, note that 
$$L=L(\mathcal{C}_n,\vd)=
\begin{pmatrix}
3&-1&0&0&\ldots&0&-1\\
-1&1&-1&0&\ldots&0&0\\
0&-1&3&-1&\ldots&0&0\\
0&0&-1&2&-1&0&0\\
0&0&0&\ddots&\ddots&\ddots&0\\
0&0&\ldots&0&-1&2&-1\\
-1&0&\ldots&0&0&-1&2\\
\end{pmatrix}.$$
Let $L_n$ be the principal submatrix of $L$ with row $n$ and column $n$ removed, i.e.,
$$L_n=
\begin{pmatrix}
3&-1&0&0&\ldots&0\\
-1&1&-1&0&\ldots&0\\
0&-1&3&-1&\ldots&0\\
0&0&-1&2&-1&0\\
0&\ldots&0&\ddots&2&-1\\
0&\ldots&0&0&-1&2\\
\end{pmatrix}.$$
By Corollary \ref{cor:interlacing} and Lemma \ref{lem:31322onpath}, we have $\mu_1(L)\geq \mu_1(L_n)>4.$
\end{proof}

\begin{proposition}\label{prop:munonlap}
Let $\vd$ be a non-Laplacian arithmetical structure on $\mathcal{C}_n$ for {$n\geq 6$}, then $\mu_1(L(\mathcal{C}_n,\vd))>4.$
\end{proposition}
\begin{proof}
We proceed by induction on $n$. Table \ref{tab:struc} shows the result is true for $n=6.$ For any $n\geq 7$, let $\vd$ determine a non-Laplacian structure on $\mathcal{C}_n$. By Proposition \ref{prop:dv=1}, there is an entry $i \in [n]$ such that $d_i=1$. Smoothing at $v_i$ we get a structure $\vd'$ in $\mathcal{C}_{n-1}$ with $d'_j=d_j-1$ if $j=i-1$ or $j=i+1$ (taken$\mod n$) and $d'_j=d_j$ for all other $j$. Let $L_i(\mathcal{C}_n,\vd)$ be the matrix $L(\mathcal{C}_n,\vd)$ with row $i$ and column $i$ removed. By Corollary \ref{cor:interlacing}, we have 
$$\mu_1(L(\mathcal{C}_n,\vd))\geq \mu_1(L_i(\mathcal{C}_n,\vd)).$$

Considering the matrix $L(\mathcal{C}_{n-1},\vd')$, we get $L(\mathcal{C}_{n-1},\vd')=L_i(\mathcal{C}_n,\vd)-B$ where $B=(b_{s,t})$ is the $(n-1)\times (n-1)$ matrix given by $b_{i-1,i-1}=b_{i,i}=b_{i-1,i}=b_{i,i-1}=1$ and $0$ entries elsewhere. By the Courant-Weyl inequalities (Lemma \ref{CW inequality}) and using the fact that the eigenvalues of $B$ are $\{2,0,0,\ldots, 0\}$, we have
$$\mu_1(L_i(\mathcal{C}_n,\vd)) \geq \mu_1(L(\mathcal{C}_{n-1},\vd'))+\mu_n(B){=} \mu_1(L(\mathcal{C}_{n-1},\vd')).$$  If $\vd'$ is not the Laplacian arithmetical structure of $\mathcal{C}_{n-1}$ then by {the inductive hypothesis}  $\mu_1(L(\mathcal{C}_{n-1},\vd'))>4$, thus
$$\mu_1(L(\mathcal{C}_n,\vd))\geq \mu_1(L_i(\mathcal{C}_n,\vd))\geq \mu_1(L(\mathcal{C}_{n-1},\vd'))>4.$$
If $\vd'$ is the Laplacian on $\mathcal{C}_{n-1}$ then, {up to symmetry}, $\vd$ is the structure $\vd=(3,1,3,2,2,\ldots,2)$. In that case, Lemma \ref{lem:31322oncycle} states that $\mu_1(L(\mathcal{C}_n,\vd))>4$.
\end{proof}

\begin{theorem}\label{thm:minimal}
Among all arithmetical structures on $\mathcal{C}_n$, the Laplacian arithmetical structure $(deg(\mathcal{C}_n) , \textbf{1})$ has minimal spectral radius.
\end{theorem}
\begin{proof}
{Table \ref{tab:struc} shows the result is true for $n\leq 6$. For $n\geq 7$, if $\vd$ is the Laplacian structure on $\mathcal{C}_n$, then by Proposition 
 \ref{prop:mulaplacian} we have $\mu_1(L(\mathcal{C}_n,\vd))\leq 4$. If $\vd$ is not the Laplacian structure, then Proposition \ref{prop:munonlap} shows $\mu_1(L(\mathcal{C}_n,\vd))>4$.}
\end{proof}
\section{Maximum spectral radius}
In this section we prove that, up to symmetry, the matrix associated to the arithmetical structure $(\vd,\vr)$ with  $\vd=(1,n+2,2,2,\cdots, 2)$ and $\vr=(n,1,2,3,\ldots, n-1)$ has the maximal spectral radius among all arithmetical structures on $\mathcal{C}_n$. The general approach is to first bound the values of the entries of the $\vd$ vector of an arithmetical structure on $\mathcal{C}_n$ by $n+2$ (Proposition \ref{prop n+2 structures}). Then we show that any structure with entries of its $\vd$ vector bounded {above} by $n+1$ will have spectral radius at most $n+2$ (Lemma \ref{lem:discard}).  We finish by showing that among the structures with some entry of its $\vd$ vector equal to $n+2$, the matrix associated to $\vd=(1,n+2,2,2,\cdots, 2)$ has the largest spectral radius and as $n\to \infty$, it approaches $n+2$ from above (Theorem \ref{thm:main}).  We will again use $\mu_1(M)$ to denote the spectral radius of a real and symmetric matrix $M$ in this section.

We  start by bounding the entries of the $\vd$ vector of an arithmetical structure on $\mathcal{C}_n$ and describing the structures that achieve the maximal possible value in their $\vd$ vector.

\begin{proposition}\label{prop n+2 structures}
Let $(\vd, \vr)$ be an arithmetical structure on $\mathcal{C}_n$. Then $d_i \leq n+2$ for all $i\in [n]$. Further, there are two types of arithmetical structures on $\mathcal{C}_n$ satisfying $d_i=n+2$ for some $i\in[n]$,
\begin{enumerate}[(a)]
    \item $\vd=(1,n+2,2,2,\ldots, 2)$ for $n \geq 3$, 
    \item $\vd^k=(1,n+2,1,\underbrace{2,2,\ldots,2}_{k\text{ times }}, 3,\underbrace{2,2,\ldots, 2}_{n-4-k \text{ times}}) \text{ for } k\in\{0,1,\ldots, n-4\}$ for $n\geq 4$,
\end{enumerate} and all their symmetries. Moreover, the arithmetical structures $\vd^k$ and $\vd^{n-4-k}$ are symmetric for $k\in\{0,1,\ldots, n-4\}$.
\end{proposition}
\begin{proof}
{We first show $d_i\leq n+2$ for all structures $\vd$ on $\mathcal{C}_n$ and all $i\in [n]$.} We proceed by induction on $n$. Table \ref{tab:struc}, shows the result holds for $n=3,4,5, 6$. Let $n\geq 7$. Given a structure $\vd$ in $\mathcal{C}_n$, if $\vd$ is the Laplacian, then $d_i=2$ for all $i$, hence the result holds. If $\vd$ is not the Laplacian, by Proposition \ref{prop:dv=1}, $d_j=1$ for some $j\in [n]$ {and $d_{j-1}\neq 1\neq d_{j+1}$, where the indices are taken mod $n$}. We can smooth at this vertex $v_j$ to obtain a structure $\vd'$ on $\mathcal{C}_{n-1}$ {(with vertex set $\{v_1,v_2,\ldots, v_n\}\setminus\{v_j\}$)} defined as 
$$d_i'=\begin{cases}
d_i &\text{ if $i\neq j-1,j+1$}\\
d_i-1& \text {if $i=j-1,j+1$}
\end{cases},$$
with indices taken modulo $n$.
Thus, $d_i\leq d_i'+1\leq (n-1)+2+1=n+2$, where the last inequality follows by the induction hypothesis. Together with $d_j=1$, this implies  $d_i\leq n+2$ for all $i\in [n]$.

To prove the second part of the theorem, if $\vd$ is an arithmetical structure on $\mathcal{C}_n$ with some $d_i=n+2$ then, by symmetry, we can assume $i=2$. By Proposition \ref{prop:dv=1}, some entry $d_j$ of $\vd$ is 1. If $j\not\in\{1,3\}$ then smoothing at $d_j$ gives a structure on $\mathcal{C}_{n-1}$ with $d_2=n+2$, a contradiction to the first part of this proof. Thus, $j=1$ or $j=3$. Smoothing at $d_j$ gives a structure $\vd'$ on $\mathcal{C}_{n-1}$ (with vertex set $\{v_1,v_2,\ldots, v_n\}\setminus\{v_j\}$) with $d_2'=n+2-1=n+1$. By the induction hypothesis, up to {symmetry},  $\vd'=(1,n+1,2,2,\ldots, 2)$ or $\vd'=(1,n+1,1,\underbrace{2,2,\ldots,2}_{k\text{ times }}, 3,\underbrace{2,2,\ldots, 2}_{n-5-k \text{ times}})$ for some $k \in\{0,1,\ldots, n-5\}$. Hence,  
$$\vd=\begin{cases}
    (1,n+2,2,2,\ldots,2,{2})& \text{ if } \vd'=(1,n+1,2,2,\ldots, 2) \text { and }j=1\\
    (1,{n+2},{1},{3},\underbrace{2,\ldots,2}_{n-4})& \text{ if } \vd'=(1,n+1,2,2,\ldots, 2) \text { and }j=3\\
    ({1},{n+2},1,\underbrace{2,\ldots,2}_{k}, 3,\underbrace{2,\ldots, 2,{2}}_{n-4-k })& \text{ if } \vd'=(1,n+1,1,\underbrace{2,\ldots,2}_{k}, 3,\underbrace{2,\ldots, 2}_{n-5-k }) \text { and }j=1\\
    (1,{n+2},{1},\underbrace{{2},\ldots,2}_{k+1}, 3,\underbrace{2,\ldots, 2}_{n-5-k})& \text{ if } \vd'=(1,n+1,1,\underbrace{2,\ldots,2}_{k}, 3,\underbrace{2,\ldots, 2}_{n-5-k}) \text { and }j=3\\
\end{cases}$$
In the first case, we get the $\vd$ vector in part $(a)$ of the statement, in the second case we get the vector $\vd^{0}$ from part $(b)$, in the third case we get the vector $\vd^k$ from part $(b)$ for $k \in \{0,1,\ldots, n-5\}$ and in the last case we get $\vd^{k+1}$ from part $(b)$ for $k \in \{0,1,\ldots, n-5\}$.

To prove the last sentence, apply the reflection to $\vd^k$ that fixes the second entry to obtain $\vd^{n-4-k}$.
\end{proof}

The next two lemmas show that any arithmetical structure whose $\vd$ vector entries  are bounded above by $n+1$ will have spectral radius of at most $n+2$. 

\begin{lemma}\label{lem:d*}
Let $n\geq 6$. Let $(\vd,\vr)$ be an arithmetical structure on $\mathcal{C}_n$ such that $d_j\leq n+1$ for all $j\in[n]$ and there is some $i\in [n]$ with $d_i=n+1$, then there is an arithmetical structure $(\vd^*,\vr^*)$ on $\mathcal{C}_n$ such that
\[d_j+d_j^*\leq n+2 \text { for all } j \in [n].\]
\end{lemma}

\begin{proof}
Up to symmetry, we can assume $i=2$. Let $\vd^*$ be the Laplacian arithmetical structure on $\mathcal{C}_{n-1}$ subdivided once at the first edge, that is, $\vd^*=(3,1,3,2,\ldots, 2)$. We have $d_2+d^*_2=n+1+1=n+2.$ It now suffices to show that $d_j\leq n-1$ for all $j\neq 2$. We show a stronger result, that $d_1\leq 3, d_3\leq 3$, and $d_j\leq 4$ for all $j\in[n]$ with $j\not \in \{1,2,3\}$. We proceed by induction on $n$. If $n=6$, Table \ref{tab:struc} shows the result holds. Let $n\geq7$ and 
\[\vd=(d_1, n+1,d_3,d_4,\ldots, d_n) \text{ with } d_j\leq n+1 \text{ for } j\in[n]. \]
Recall there is an index $\ell$ such that $d_\ell=1$. If $\ell\neq 1$ and $\ell\neq 3$, then smoothing at the vertex $v_\ell$ gives the structure 
 \[\vd'=(d_1,n+1,d_3,\ldots, d_{\ell-2},d_{\ell-1}-1,d_{\ell+1}-1,d_{\ell+2},\ldots, d_n)\] on $\mathcal{C}_{n-1}$. Since this structure has $d'_2=n+1$, then by Proposition \ref{prop n+2 structures}, we have $d_1\leq 2$, $d_3\leq 2$, $d_{\ell-1}-1\leq 3$, $d_{\ell+1}-1\leq 3$ and $d_j\leq 3$ for $j \not \in \{1,2,3,\ell-1,\ell+1\}$. Thus, we get our desired bounds.  

If  $\ell=3$ (so $d_3=1)$ then smoothing at vertex $v_3$ gives you the structure $\vd'=(d_1,n,d_4-1,d_5,\ldots, d_n)$ on $\mathcal{C}_{n-1}$. By the inductive assumption, $d_1\leq 3$, $d_4-1\leq 3$, and $d_i\leq 4$ for $ i \in \{5,6,\ldots, n\}$. Hence, we get our desired bounds. The case $\ell_1=1$ (so $d_1=1$) is similar to the case $\ell=3$.
\end{proof}

\begin{lemma}\label{lem:discard}
Let $(\vd, \vr)$ be an arithmetical structure on $\mathcal{C}_n$. If $d_i \leq n+1$ for all $i \in \{1,2,\ldots, n\}$, then $\mu_1(L(\mathcal{C}_n, \vd)) \leq n +2$.
\end{lemma}
\begin{proof}
If $d_i\leq n$ for all $i\in \{1,2,\ldots, n\}$ then by combining Corollary \ref{cor:absolute value} and Lemma \ref{prop:rowsum}, $\mu_1(L(\mathcal{C}_n,\vd))\leq n+2$. If $d_i=n+1$ for some $i$, let $L'$ be the matrix whose entries are the absolute value of the entries of $L=\diag(\vd)-A_{\mathcal{C}_n}$, that is, $L'=\diag(\vd)+A_{\mathcal{C}_n}$. Thus, $\mu_1(L)\leq \mu_1(L')$ by Corollary \ref{cor:absolute value}. It now suffices to show $\mu_1(L')\leq n+2$.

Let $I_n$ be the $n\times n$ identity matrix and $X=(n+2)I_n-L'=(n+2)I_n-\diag(\vd)-A_{\mathcal{C}_n}$. By Lemma \ref{lem:d*}, there is a structure $\vd^*$ of $\mathcal{C}_n$ such that $d_i+d^*_i\leq n+2$ for all $i \in [n]$ so $d^*_i\leq n+2-d_i$. Then,
\[X=(n+2)I_n-\diag(\vd)-A_{\mathcal{C}_n}=\diag(\vd^*)-A_{\mathcal{C}_n}+\diag(\textbf{y})=L(\mathcal{C}_n,\vd^*)+\diag(\textbf{y})\]
for the vector $\textbf{y}$ with non-negative entries $(n+2-d_i-d_i^*)_{i \in[n]}$. By Proposition 1.1 in \cite{L89}, $L(\mathcal{C}_n,\vd^*)$ is an $M$-matrix, thus $\mu_n(L(\mathcal{C}_n,\vd^*))\geq 0.$ Since $\diag(\textbf{y})$ has non-negative entries, $\mu_n(X)\geq 0$ by Lemma \ref{CW inequality}. Using this and the fact that for any eigenvalue $\lambda$ of $L'$ there is an eigenvalue $n+2-\lambda$ of $X$, we get $\mu_1(L')\leq n+2.$
\end{proof}

We are now ready to prove the main theorem of this article. 

\begin{theorem}\label{thm:main}
Of the arithmetical structures on $\mathcal{C}_n$, up to symmetry, the structure $\mathbf{d} = (1,n+2, 2,2,\cdots, 2)$ has maximum spectral radius which tends to $n+2$ as $n\to\infty$.
\end{theorem}

\begin{proof}
As in Proposition \ref{prop n+2 structures}, define
\begin{itemize}
    \item $\vd=(1,n+2,2,2,\ldots, 2)$ for $n \geq 3$, 
    \item $\vd^k=(1,n+2,1,\underbrace{2,2,\ldots,2}_{k\text{ times }}, 3,\underbrace{2,2,\ldots, 2}_{n-4-k \text{ times}}) \text{ for } k\in\{0,1,\ldots, n-4\}$ for $n\geq 4$.
\end{itemize}
Let $M = L(\mathcal{C}_n, \mathbf{d})$ and $M^{(k)} = L(\mathcal{C}_n, \mathbf{d}^k)$, that is
\[M= \begin{pmatrix}
     1&-1&0&0&\ldots&0&-1  \\
     -1&n+2&-1&0&\ldots&0&0  \\
      0&-1&2&-1&0&\ldots&0  \\
      0&0&-1&2&-1&\ldots&0  \\
      0&0&\ddots&\ddots&\ddots&\ddots&0  \\
      0&0&\ddots&&-1&2&-1  \\
      -1&0&&&0&-1&2  \\
\end{pmatrix}\]
and 
\[M^{(k)}= \begin{pmatrix}
     1&-1&0&0&\ldots&0&-1  \\
     -1&n+2&-1&0&\ldots&0&0  \\
      0&-1&1&-1&0&\ldots&0  \\
      0&0&-1&\ddots&-1&\ldots&0  \\
      0&0&\ddots&\ddots&3&\ddots&0  \\
      0&0&\ddots&&-1&\ddots&-1  \\
      -1&0&&&0&-1&2  \\
\end{pmatrix}.\]
First, note that each of these matrices has spectral radius at least $n+2$ by considering the Rayleigh quotient with vector $\mathbf{e}_2$. We first show that $\mu_1(M) > \mu_1(M^{(k)})\geq n+2$ for each $k$. By Lemma \ref{lem:discard} and 
Proposition \ref{prop n+2 structures}, this implies that the arithmetical structure that maximizes the spectral radius is  $\vd$. 
 We then complete the proof by showing that $\mu_1(M) \leq n+2 + \frac{24}{n}$. Before tightening the upper bound we record that the spectral radius of $M$  is at most $n+4$ by combining Corollary \ref{cor:absolute value} and Lemma \ref{prop:rowsum}.

First, fix $k\in \{0,\ldots, \lceil \frac{n-4}{2}\rceil\}$ and we will show that $\mu_1(M) > \mu_1(M^{(k)})$. Let $\mathbf{x}$ be an eigenvector for $\mu_1(M^{(k)})$ and normalize so that it has infinity norm equal to $1$. By Lemma \ref{lem:rayleigh} we have that 
\[
\mu_1(M) \geq \frac{\mathbf{x}^T M\mathbf{x}}{\mathbf{x}^T\mathbf{x}} \quad \mbox{and} \quad \mu_1(M^{(k)}) =\frac{\mathbf{x}^T M^{(k)}\mathbf{x}}{\mathbf{x}^T\mathbf{x}},
\]
thus it suffices to show that 
\[
\frac{\mathbf{x}^T (M-M^{(k)})\mathbf{x}}{\mathbf{x}^T\mathbf{x}} > 0.
\]

Note that the matrix $M-M^{(k)}$ has only $2$ nonzero entries: a $1$ in the third diagonal entry and a $-1$ in the $(k+4)^{th}$ diagonal entry. Hence we have that $\mathbf{x}^T (M-M^{(k)})\mathbf{x} = \mathbf{x}_3^2 - \mathbf{x}_{k+4}^2$, and so it suffices to show that $|\mathbf{x}_3| > |\mathbf{x}_{k+4}|$. We give a lower bound for $|\mathbf{x}_3|$ and an upper bound for $|\mathbf{x}_{k+4}|$. To do this, the eigenvector-eigenvalue equation gives
\[
-\mathbf{x}_{i-1} + (\mathbf{d}^k)_i \mathbf{x}_i - \mathbf{x}_{i+1} = \mu_1(M^{(k)}) \mathbf{x}_i,
\]
for $i\in [n]$, where indices are computed modulo $n$. By the normalization, we have that \[|\mathbf{x}_i(\mu_1(M^{(k)}) - (\vd^k)_i)| \leq 2 \]  and hence for all $i\in[n]\setminus\{2,k+4\}$ we have that $|\mathbf{x}_i| \leq \frac{2}{n}$ and $|\mathbf{x}_{k+4}| \leq \frac{2}{n-1}$, using that $\mu_1(M^{(k)}) \geq n+2$. This also implies that $\mathbf{x}_2 = 1$. Iterating the argument, we have that 
\[
\mathbf{x}_i = \frac{-\mathbf{x}_{i-1} - \mathbf{x}_{i+1}}{\mu_1(M^{(k)}) - (\vd^k)_i} .
\]
And so we have that 
\begin{equation}\label{eigenvector entry upper bound}
|\mathbf{x}_{k+4}| \leq \frac{\frac{4}{n}}{\mu_1(M^{(k)}) - 3} \leq \frac{4}{n(n-1)},
\end{equation}
using that $\mu_1(M^{(k)}) \geq n+2$. Finally, for $i = 3$, we have 
\begin{equation}\label{eigenvector entry lower bound}
|\mathbf{x}_3| \geq \frac{1-\frac{2}{n}}{\mu_1(M^{(k)}) - 1} \geq \frac{1}{n+1} - \frac{2}{n(n+3)},
\end{equation}
using that $\mathbf{x}_2 = 1$ and $n+2 \leq \mu_1(M^{(k)})\leq n+4$. Combining Equations \eqref{eigenvector entry upper bound} 
 and \eqref{eigenvector entry lower bound} gives that 
${|\mathbf{x}_3|\geq |\mathbf{x}_{k+4}|}$ for $n\geq 7$, thus $\mu_1(M) > \mu_1(M^{(k)})$ for $n\geq 7$. The result for $n\leq 6$ can be read from Table \ref{tab:struc}.

Finally, to show that $\mu_1(M) \to n+2$ as $n\to \infty$, let $\mathbf{z}$ be an eigenvector for $\mu_1(M)$. Similar to the previous argument, we may normalize $\mathbf{z}$ so that it has maximum entry equal to $1$, and the same argument as before yields that for $i\not=2$ we have that $|\mathbf{z}_i| \leq \frac{2}{n}$ and that $\mathbf{z}_2 = 1$. Using Lemma \ref{lem:rayleigh} we compute the Rayleigh quotient with eigenvector $\mathbf{z}$,
\[
{ \mu_1(M) = \frac{\mathbf{z}^TM\mathbf{z}}{\mathbf{z}^T\mathbf{z}} = \frac{\sum_{i=1}^n \mathbf{d}_i \mathbf{z}_i^2 - 2\sum_{i=1}^n \mathbf{z}_i \mathbf{z}_{i+1}}{\mathbf{z}^T\mathbf{z}} \leq \frac{n+2 + 2(n-1)\frac{4}{n^2} +2( 2\cdot \frac{2}{n} + (n-2)\cdot \frac{4}{n^2})}{1}.}
\]
Hence, we have $\mu_1(M) \leq n+2 + \frac{24}{n}$.
\end{proof}

\section{Concluding Remarks}
Given the results presented in \cite{W21} for path graphs and here for cycle graphs, a natural question to ask is the following.
\begin{question}\label{q1}
Given a graph $G$, which arithmetical structures maximize the spectral radius of $\diag(\vd)-A$? Which structures minimize it?
\end{question}
As pointed out by \cite[Remark 4.2]{W21}, the Laplacian structure is not always the structure that minimizes the spectral radius. 
One difficulty in answering this question for families of graphs other than path graphs and cycle graphs is that we do not have a complete list of structures or an enumeration for them. Perhaps a natural starting point is to consider Question \ref{q1} on stars or complete graphs because of the connection to Egyptian fractions.  

The number of arithmetical structures is known for path graphs and cycle graphs \cite{B18}. For path graphs, reflecting the graph across its middle vertex (if $n$ is odd) or its middle edge (if $n$ is even) preserves the neighbor set of each vertex. Thus, we can call two structures $(\vd,\vr)$ and $(\vd',\vr')$ symmetric if  
\[d'_k=d_{n+1-k} \text{ and } r'_k=r_{n+1-k}.\]
\begin{question}
How many arithmetical structures on $\mathcal{P}_n$ are there up to symmetry?
\end{question} 
Similarly, we call two structures $(\vd,\vr)$ and $(\vd',\vr')$  on $\mathcal{C}_n$ symmetric if they are in the same orbit of the action of $D_{2n}$ on the set of arithmetical structures on $\mathcal{C}_n$, as explained in Section \ref{subsec:arith}.
\begin{question}
How many arithmetical structures on $\mathcal{C}_n$ are there up to symmetry?
\end{question} 

\section*{Acknowledgments}
The authors would like to thank Dino Lorenzini for interesting remarks and questions. A. Diaz-Lopez's research is supported in part by National Science Foundation grant DMS-2211379. M. Tait's research is supported in part by National Science Foundation grant DMS-2011553.
\bibliography{Ref}
\bibliographystyle{amsplain}
\end{document}